\documentclass{amsart}
\usepackage{hyperref}
\usepackage[T1]{fontenc}
\usepackage{times}
\usepackage{amssymb,verbatim}
\theoremstyle{plain}
\usepackage[T1]{fontenc}
\usepackage[utf8]{inputenc}
\usepackage[english]{babel}
\usepackage{amsthm}
\usepackage{graphics}
\usepackage{amsmath}
\usepackage{amstext}
\usepackage{multirow}
\usepackage{graphicx}
\usepackage{amsmath}
\usepackage{amssymb}
\usepackage{amsthm}
\usepackage{amstext}
\usepackage{mathrsfs}
\usepackage{amscd}
\usepackage{mathtools}
\usepackage{tikz}
\usetikzlibrary{matrix}

\usepackage{color}

\newtheorem{thm}{Theorem}[section]

\newtheorem{lemma}[thm]{Lemma}
\newtheorem{prop}[thm]{Proposition}

\newtheorem{cor}[thm]{Corollary}

\newtheorem{THM}{Theorem}

\theoremstyle{remark}
\newtheorem{example}[thm]{Example}
\newtheorem{remark}[thm]{Remark}

\newcommand{\mb}{\mathbb}
\newcommand{\mc}{\mathcal}

\newcommand{\C}{\mb C}

\newcommand{\F}{\mc F}
\newcommand{\G}{\mc G}

\DeclareMathOperator{\id}{id}

\DeclareMathOperator{\codim}{codim}

\DeclareMathOperator{\Aut}{Aut}

\DeclareMathOperator{\Mor}{Mor}
\DeclareMathOperator{\ev}{ev}
\DeclareMathOperator{\Aff}{Aff}

\newcommand{\KF}{{K_\F}}
\newcommand{\TF}{{T_\F}}

\newcommand{\Ftang}{\F_{\rm tang}}

\numberwithin{equation}{section}

\usepackage{mathtools}
\usepackage{ifmtarg}

\addtocounter{section}{0}             
\numberwithin{equation}{section}       

\title{Deformation of rational curves along foliations}
\author{Frank Loray}
\address{ Univ Rennes, CNRS, IRMAR - UMR 6625, F-35000 Rennes, France\hfill}
\email{frank.loray@univ-rennes1.fr}

\author{Jorge Vit\'orio Pereira}
\address{IMPA, Estrada Dona Castorina, 110, Horto, Rio de Janeiro,
Brasil}
\email{jvp@impa.br}

\author{Fr\'ed\'eric Touzet}
\address{ Univ Rennes, CNRS, IRMAR - UMR 6625, F-35000 Rennes, France\hfill}
\email{frederic.touzet@univ-rennes1.fr}



\begin{document}

\begin{abstract}
Deformation of morphisms along leaves of foliations define the
tangential foliation on the corresponding space of morphisms.
We prove that codimension one foliations having a tangential foliation with at least one non-algebraic leaf are transversely homogeneous with structure group determined by the codimension of the non-algebraic leaf in its Zariski closure. As an application, we provide a structure theorem for degree three foliations on $\mathbb P^3$.
\end{abstract}

\maketitle

\setcounter{tocdepth}{1}


\section{Introduction}

\subsection{Motivation}
Singular holomorphic codimension one foliations on projective spaces of dimension at least three have been widely studied in recent years. Much of the recent activity on the subject was spurred by the classification of irreducible components of the space of foliations of degree two by Cerveau and Lins Neto \cite{MR1394970}.
Despite of the growing literature on the subject, not much of it is devoted to the next simplest case: irreducible components of the space of foliations of degree three on $\mathbb P^3$. A notable exception is \cite{MR3088436} where it is proven that they  are either transversely affine foliations, or are rational pull-backs of foliations on surfaces. In this paper we refine  \cite{MR3088436} by means of the following result.

\begin{THM}\label{THM:A}
If  $\F$ is a codimension one singular holomorphic foliation on $\mathbb P^3$ of degree three then
\begin{enumerate}
\item $\F$ is defined by a closed rational $1$-form without codimension one zeros; or
\item there exists a degree one foliation by algebraic curves  tangent to $\F$; or
\item $\F$ is a linear pull-back of a degree $3$ foliation on $\mathbb P^2$; or
\item $\F$ admits a rational first integral.
\end{enumerate}
\end{THM}

The foliations defined by closed rational $1$-forms without codimension one zeros lie at the so called logarithmic components of the space of foliations, see \cite[Section 2.5]{MR3066408} or \cite[Lemma 8]{MR1394970}. It is also known the existence of irreducible components, for the space of codimension one foliations of every degree on projective spaces of any dimension $n\ge 3$, with general element equal to a linear pull-back of a foliation on $\mathbb P^2$. Theorem \ref{THM:A} reduces the problem of determining/classifying the irreducible
components of the space of degree three foliations on $\mathbb P^3$ to the problem of determining/classifying irreducible components with general element fitting into the descriptions given by items (2) and (4) of Theorem \ref{THM:A}.

Theorem \ref{THM:A} is proved through the study of deformations of morphisms from $\mathbb P^1$ to $\mathbb P^3$ along  leaves of the codimension one foliation.

\subsection{Tangential foliations}  As it was observed in   \cite{2011arXiv1107.1538L}, foliations on a given uniruled manifold $X$ naturally induce foliations on the space of morphisms $\Mor(\mathbb P^1,X)$, the so called tangential foliations. A leaf $L$ of a tangential foliation correspond to a maximal analytic family of morphisms which do not move the points outside the starting leaves. In other words, if $L$ is a leaf of a tangential foliation and $L \times \mathbb P^1 \to X$ is the restriction of the evaluation morphism then the pull-back of the foliation on $X$ to $L \times \mathbb P^1$ is the foliation given by the fibers of the projection $L \times \mathbb P^1 \to \mathbb P^1$,  or the foliation with just one leaf. A more precise definition of  tangential foliation  can be found in  Section \ref{S:tangential}.

In this paper, we refine  our previous result on the subject \cite[Theorem 6.5]{2011arXiv1107.1538L}.

\begin{THM}\label{THM:iff}
Let $X$ be a uniruled projective manifold
and  let $\mathcal F$ be a codimension one foliation on $X$. Fix an irreducible component $M$ of
the space of morphisms from $\mathbb P^1$ to $X$ containing a free morphism and let $\Ftang$ be
the tangential foliation of $\mathcal F$ defined on this irreducible component $M \subset \Mor(\mathbb P^1,X)$.
Let $\delta$ be the codimension of a general leaf $L$ of $\Ftang$ inside its Zariski closure $\overline L$,
i.e.  $\delta = \dim \overline{L} - \dim L$. Then $\delta \le 3$. Furthermore, if $L$ is not algebraic (i.e. $\delta>0$),
    then the following assertions hold true.
	\begin{enumerate}
		\item $\delta=3$ if and only if $\F$ is transversely projective but not transversely affine; and
		\item $\delta=2$ if and only if $\F$ is transversely affine but not virtually transversely Euclidean; and
		\item $\delta=1$ if and only if $\F$ is virtually transversely Euclidean.
	\end{enumerate}
\end{THM}

We refer the reader  to Section \ref{S:hierarchy} for the definition of transversely projective, transversely affine, and virtually transversely Euclidean foliations.

Whereas \cite[Theorem 6.5]{2011arXiv1107.1538L} provides necessary conditions for a foliation to be transversely projective/transversely affine/virtually transversely Euclidean;
 Theorem \ref{THM:iff} above  establishes  an \emph{equivalence}
between the codimension $\delta$ and the dimension of the corresponding transverse Lie algebra: $\mathfrak{sl}_2$ for transversely projective, $\mathfrak{aff}$ for transversely affine, and $\mathbb C$ for virtually transversely Euclidean. Even more importantly, Theorem \ref{THM:iff} improves on the previous result by dropping the hypothesis on the Zariski denseness of the general leaf of the tangential foliation. This is achieved by means of Theorem \ref{T:extension}, a result which might has independent interest.

\subsection{Bootstrapping}
On simply connected uniruled manifolds carrying rational curves in general position with respect to the foliation in question, one can push further the analysis carried out to prove Theorem \ref{THM:iff} in order to achieve the more precise result below. Its proof consists in successive applications of Theorem \ref{THM:iff}, where in each step one gains more constraints on the nature of transverse structure.

\begin{THM}\label{THM:bootstrapping}
Let $X$ be a simply connected uniruled projective manifold
and  let $\mathcal F$ be a codimension one foliation on $X$.
Fix an irreducible component $M$ of the space of morphisms $\Mor(\mathbb P^1,X)$ containing a free morphism and let $\Ftang$ be the tangential foliation of $\F$ defined on $M$.  If the general leaf of $\Ftang$ is not algebraic and  the general morphism $f: \mathbb P^1 \to X$ in $M$  intersects non-trivially and transversely all the algebraic hypersurfaces invariant by $\F$ then
$\F$ is defined by a closed rational $1$-form without divisorial components in its zero set.
\end{THM}

\subsection{Structure of the paper} In Section \ref{S:hierarchy} we recall the basic definitions concerning transverse structure of codimension one foliations and present the key properties which will be used throughout. Section \ref{sec:tomography} is devoted to the proof of extension of transverse structures from fibers of a fibration to the whole ambient manifold. Section \ref{S:tangential} studies the tangential foliations and contain the proofs of Theorems \ref{THM:iff} and \ref{THM:bootstrapping}. Finally, Theorem \ref{THM:A} is proved in Section \ref{S:3}.

\subsection{Acknowledgements}This work was partially supported by the Brazil-France Agreement in Mathematics. F. Loray and F. Touzet acknowledge the support of  ANR-16-CE40-0008 project Foliage and
MATH-AmSud project n\textsuperscript{\underline{\scriptsize o}}88881.117598.  J. V. Pereira acknowledges the support of  CNPq, Faperj,  and Freiburg Institute for Advanced Studies (FRIAS). The research leading to these results has received funding from the People Programme (Marie Curie Actions) of the European Union's Seventh Programme (FP7/2007-2013) under REA grant agreement n\textsuperscript{\underline{\scriptsize o}} [609305].


\section{Transverse structures for codimension one foliations}\label{S:hierarchy}

\subsection{Transversely affine and virtually transversely Euclidean foliations}
Let $\F$ be a codimension one  foliation on a projective manifold $X$. Since any line bundle on $X$ admits rational sections, we may choose  a rational $1$-form $\omega_0$ defining $\F$. The foliation $\F$ is  transversely affine if
there exists a  rational $1$-form $\omega_1$ such that
\[
d\omega_0  = \omega_0 \wedge \omega_1  \quad \text{ and } \quad     d\omega_1  = 0 \, .
\]
We say that the pair $(\omega_0,\omega_1)$ defines a transverse affine structure for $\F$.  Although this definition is made in terms of a particular $1$-form $\omega_0$ defining $\F$, if $\omega_0' = h \omega_0$ is another rational $1$-form defining  $\F$ then $\omega_1' = \omega_1 - \frac{dh}{h}$ is a closed rational $1$-form which satisfies $d\omega_0'= \omega_0'\wedge \omega_1'$.
It is therefore natural to say that two pairs of rational $1$-forms $(\omega_0,\omega_1)$ and $(\omega_0', \omega_1')$ define the same
transverse affine structure for $\F$ if there exists a non-zero rational function $h$ such that
$(\omega_0',\omega_1') = (h \omega_0, \omega_1 - \frac{dh}{h})$.

A transverse affine structure for $\F$ determines   a local system of  first integrals  given by
the branches of the multi-valued function
\[
   F =  \int \left(\exp\int \omega_1 \right) \omega_0 \, .
\]
Different branches of $F$ at a common domain of definition
differ by  left composition with an element of the
affine  group $\Aff(\C) \simeq \mathbb C \rtimes \mathbb C^*$.

If $\omega_1$ is logarithmic with  all its periods  integral multiples of $2\pi \sqrt{-1}$
then  $(\exp\int \omega_1) \omega_0$ is a closed rational $1$-form defining $\F$. In this case we will say
that the transverse affine structure defined by $(\omega_0,\omega_1)$ is transversely Euclidean.

Transverse affine structures for which $\omega_1$ is logarithmic with all its  periods commensurable to $2\pi \sqrt{-1}$
are called virtually transversely Euclidean structures. In this case, the $1$-form $(\exp\int \omega_1) \omega_0$
is not necessarily a rational $1$-form, but after passing to a finite ramified covering of $X$ it becomes
one.

\subsection{Transversely projective foliations}
Similarly, a foliation  $\F$ on $X$ is called transversely projective if for any rational $1$-form $\omega_0$ defining $\F$
there exists rational $1$-forms $\omega_1$ and $\omega_2$ such that
\begin{align*}
d\omega_0 & = \omega_0 \wedge \omega_1  \\
d\omega_1 & = \omega_0 \wedge \omega_2 \\
d\omega_2 & =\omega_1 \wedge \omega_2 \, .
\end{align*}
The triple of rational $1$-forms $(\omega_0, \omega_1,\omega_2)$ defines a projective structure
for $\F$. Two triples $(\omega_0, \omega_1, \omega_2)$ and $(\omega_0',\omega_1',\omega_2')$
define the same projective structure for $\F$ if there exists rational functions $f, g \in \mathbb C(X)$
such that
 \begin{align*}
\omega_0' & = f \omega_0  \\
\omega_1' & = \omega_1 - \frac{df}{f} + g \omega_0 \\
\omega_2' & = f^{-1} \omega_2 + g \omega_1 + g^2 \omega_0 - dg \, .
\end{align*}
As in the case of transversely affine foliations, the transversely projective foliations admit a
canonical collection of local holomorphic first integrals defined on the complement of the polar set
of $\omega_0, \omega_1, \omega_2$, see for instance \cite{MR1955577}.

For a thorough discussion about transversely affine and transversely
projective foliations of codimension one on projective manifolds including a description
of their global structure we invite the reader  to consult \cite{MR3294560} and \cite{MR3522824}  respectively. Here we will review two important features of transversely homogeneous foliations/structures of codimension one  in a
formulation slightly more general than what is currently available in the literature. These two features  play an essential role in this paper.

\subsection{Behaviour under dominant rational maps}
Transversely homogenous structures for codimension one foliations behave rather well
with respect to dominant rational maps as the result below shows.

\begin{lemma}\label{L:wellknown}
	Let $\G$ be a codimension one  foliation on a projective manifold $Y$, $\F$ a codimension one foliation on a projective manifold $X$, and $F: Y \dashrightarrow X$ a dominant rational map such that $\G=F^* \F$. If $\G$ is
	transversely projective, transversely affine, virtually transversely Euclidean then $\F$ is, respectively, transversely projective, transversely affine, virtually transversely Euclidean.
\end{lemma}

\begin{proof}
	This is essentially the content  of  \cite[Lemma 6.2]{2011arXiv1107.1538L}, a geometric translation of  \cite[Lemma 2.1 and Lemma 3.1]{MR1955577}.
    The only part of the statement not explicit proved there is that when $\G$ is virtually transversely Euclidean then the same holds true for  $\F$.
	For the sake of completeness let us present  a proof of this implication.
	
	Notice that we can assume from the beginning that $\F$ does not admit  a rational first integral, as otherwise there would be nothing to prove.
    After cutting $Y$ with hyperplane sections we can assume that $Y$ and $X$ have the same dimension.
    We can further assume that $\G$ is transversely Euclidean and that the generically finite rational map
    $F: Y \dashrightarrow X$ is Galois. Let $\omega$ be a rational $1$-form defining $\F$. Let $\eta$ be a closed rational $1$-form defining $\G$.
    If $\varphi$ is one of the deck transformation of $F$ then
	$\varphi^* \eta = \lambda \eta$ for some root of the unity $\lambda$. Since $\eta$ and $F^* \omega$ define the same foliation,
    there exists a rational function $h$ such that $F^*\omega = h \eta$.  Of course $\varphi^* h = \lambda^{-1} h$ and consequently $\varphi^* \frac{dh}{h} = \frac{dh}{h}$.
    The closedness of $\eta$ implies
	\[
		d F^* \omega = \frac{dh}{h} \wedge F^*\omega.
	\]
	 The invariance of $\frac{dh}{h}$ under the deck transformations allows us to find a logaritmic $1$-form
	 $\beta$ on $X$ with rational residues such that $\frac{dh}{h} = F^* \beta$. Hence $d\omega  = \beta \wedge \omega$
    and $\F$ is virtually transversely Euclidean.
\end{proof}

The converse statement is trivially true: if $\mathcal F$ is transversely projective,
transversely affine, or virtually transversely Euclidean,
then the same holds true for $\mathcal G$. Indeed, it suffices to pull-back the $1$-forms $\omega_i$ defining the structure
by $F$.

\subsection{Uniqueness of tranversely homogeneous structures}
Transversely homogeneous structures for codimension one foliations are
unique except for very special exceptions which are described by the next
result.

\begin{lemma}\label{L:uniqueness}
	Let $\F$ be a codimension one foliation on a projective manifold.
	The following assertions hold true.
	\begin{enumerate}
		\item If $\F$ admits two non-equivalent transversely projective structures then $\F$ admits a virtually transversely Euclidean structure.
		\item If $\F$ admits two non-equivalent transversely affine structures then $\F$ is defined by a closed rational $1$-form, i.e. $\F$ admits a transversely Euclidean structure.
		\item If $\F$ admits two non-equivalent virtually transversely Euclidean structure then $\F$ admits a rational first integral.
	\end{enumerate}
\end{lemma}
\begin{proof}
The first two possibilities are described  in \cite[Lemma 2.20]{MR2324555}. For describing the last possibility, assume that $\F$ is defined
a rational $1$-form $\omega_0$ and observe that the existence
of two non-equivalent virtually transversely Euclidean structures is equivalent to the existence of two linearly independent logarithmic $1$-forms
$\omega_1, \omega_1'$, both  with periods commensurable to $\pi i$ and satisfying
\[
    d \omega_0 = \omega_0 \wedge \omega_1 = \omega_0 \wedge \omega_1' \, .
\]
Their difference $\omega_1 - \omega_1'$ is non-zero, proportional to $\omega_0$, and a suitable complex multiple of it has
all its periods in $2 \pi i \mathbb Z$. In other words,  there  exists   a constant $\lambda \in \mathbb C^*$
such that
\[
    \exp \left( \int \lambda ( \omega_1 - \omega_1') \right)
\]
is a rational function constant along the leaves of $\F$.
\end{proof}


\section{Extension of transverse structures}\label{sec:tomography}
The purpose of this section is to prove the following result.

\begin{thm}\label{T:extension}
	Let $\F$ be a codimension one foliation on a projective manifold $X$. Suppose there exists a projective variety $B$ and a morphism $f : X \to B$ with general fiber irreducible  such that the $\F$ is transversely projective/transversely affine/virtually transversely Euclidean when restricted to a very general fiber of $f$. If the restriction of $\F$ to the very general fiber  does not admit a rational first integral then $\F$ is, respectively, transversely projective/transversely affine/virtually transversely Euclidean.
\end{thm}

In the statement of Theorem \ref{T:extension} one can replace the existence of the morphism $f$ by the existence
of a covering  family $\mathscr Z$ of subvarieties of  $X$  such that  $\F$ is transversely projective/transversely affine/virtually transversely Euclidean
when restricted to a very general element of the family in order to achieve the same conclusion. One reduces to the statement above
by pulling back the foliation to the total space of the family of cycles, applying Theorem \ref{T:extension} to this pull-back foliation, and
descending the conclusion to $X$ using Lemma \ref{L:wellknown}.

 Before proceeding to the proof of Theorem \ref{T:extension} we point out that the  projectiveness of $X$ is essential for the validity of the Lemma \ref{L:persistence} as the example below shows.

\begin{example}
	Let $F = E\times E$ be the square of an elliptic curve and consider the automorphism of it defined
	by $\varphi(x,y) = (2x + y , x +y)$. This is the standard example of an hyperbolic automorphism. It leaves invariant two linear foliations on $E\times E$, say $\F_+, \F_-$, defined by closed holomorphic $1$-forms
	$\omega_+, \omega_-$ respectively. As can be easily seen $\varphi^* \omega_{\pm} = \lambda_{\pm} \omega_{\pm}$  where $\lambda_{\pm}$ are the eigenvalues of the matrix
	\[
	\left(\begin{matrix}
	2&1\\1&1
	\end{matrix}\right).
	\]
	If we now consider  a projective curve $B$  and a suspension of
	a representation $\pi_1(B) \to \Aut(\F_{\pm})$, containing $\varphi$ in its image,  we get a genuinely transversely affine foliation
	on compact complex manifold $X$ fibering over $B$ with fibers $F$ such that the restriction to any fiber is transversely Euclidean and without rational first integral.
	The ambient manifold $X$ is not K\"{a}hler. Indeed, aiming at a contradicition, suppose that $X$ is K\"{a}hler.  Then $\varphi$ has to preserve a K\"{a}hler class and  \cite[Proposition 2.2]{MR521918} implies  that some non trivial power of $\varphi$ is a translation. This gives the sought contradiction.
\end{example}

\subsection{Proof of Theorem \ref{T:extension}}
Theorem \ref{T:extension} follows from the  next three results combined.

\begin{lemma}\label{L:preparation}
	Let $\F$ be a codimension one foliation on a projective manifold $X$. Assume that $X$ is endowed with a fibration $f : X \to B$ such  that the restriction of $\F$ to the very general fiber of $f$  is
	transversely projective. Then there exists rational $1$-forms
	$\omega_0, \omega_1, \omega_2$ and $\omega_3$ with the following properties.
	\begin{enumerate}
		\item The $1$-forms $\omega_0, \ldots, \omega_3$ satisfy the system of equations
	\begin{equation}\label{eq:TruncatedGodbillonVey}
	\left\{\begin{matrix}
		d\omega_0=\omega_0\wedge\omega_1\hfill\\
		d\omega_1=\omega_0\wedge\omega_2\hfill\\
		d\omega_2=\omega_0\wedge\omega_3+\omega_1\wedge\omega_2\\
	\end{matrix}\right.
	\end{equation}
	\item The foliation $\F$ is defined by $\omega_0$.
	\item The restriction of $1$-form $\omega_3$ to a general fiber  of $f$ is zero.
	\item If we further assume that the restriction of $\F$ to a very general fiber of $f$ is transversely affine then
	we have that both $\omega_2$ and $\omega_3$ restrict to zero at a general fiber of $f$.
	\end{enumerate}
\end{lemma}
\begin{proof}
	We start by choosing $\alpha_0, \alpha_1, \alpha_2, \alpha_3$ rational $1$-forms on $X$ such that
	$\alpha_0$ defines $\F$ and which satisfy the system equations (\ref{eq:TruncatedGodbillonVey}). The existence of such $1$-forms is well-known and can be traced back to Godbillon-Vey, see \cite[Section 2.1]{MR2324555}.

	Let $\G$ be the restriction  of $\F$ to the (schematic) generic fiber $\mathscr X$ of  $f$.
    Hence $\mathscr X$  is a projective manifold defined over the function field $\mathbb C(B)$ of the basis of the fibration
    and $\G$ is  a codimension one foliation on $\mathscr X$. The algebraic nature of transversely projective structures
    implies that $\G$ is transversely projective. If we denote by $\beta_0, \ldots, \beta_3$ the
    restrictions of $\alpha_0, \ldots, \alpha_3$ to $\mathscr X$ then, according to \cite[Section 2.3]{MR2324555},
    we can replace $\beta_2$ by 	$\beta_2 + f \beta_0$ where $f$ is a rational function on $\mathscr X$ in such a way that   they now satisfy
	\[ \left\{
        \begin{matrix}
	       d\beta_0=\beta_0\wedge\beta_1\hfill\\
	       d\beta_1=\beta_0\wedge\beta_2\hfill\\
	       d\beta_2=\beta_1\wedge\beta_2\\
	    \end{matrix}\right.
	\]
	
	Consider a lift of $f$ to a rational function on $X$ (still denoted by $f$) and
	set $\omega_0= \alpha_0$, $\omega_1 = \alpha_1$, $\omega_2 = \alpha_2 + f \alpha_0$.
	According to \cite[Corollary 2.4]{MR2324555} there exists a unique $\omega_3$ such that
	the system of equations (\ref{eq:TruncatedGodbillonVey}) is satisfied. Moreover, the uniqueness
	of $\omega_3$ implies that its restriction to the generic fiber of $f$ must be zero. This is sufficient
	to prove item (3) of the lemma. The proof if item (4) is completely similar.
\end{proof}

    \begin{prop}
	   Let $\F$ be a codimension one foliation on a projective manifold $X$.
       Assume that $X$ is endowed with a fibration $f : X \to B$ and that the restriction of $\F$ to a general fiber of $f$ does not admit a rational first integral.
       If there exists $1$-forms $\omega_0, \omega_1, \omega_2, \omega_3$ satisfying item (3) of Lemma \ref{L:preparation} then $\mathcal F$ is transversely projective.
       Similarly, if there exists $1$-forms $\omega_0, \omega_1, \omega_2, \omega_3$ satisfying item (4)  of Lemma \ref{L:preparation} then $\mathcal F$ is transversely affine.
    \end{prop}
    \begin{proof}
       Assume that the dimension of $B$ is $q$ and let $f_1,\ldots, f_q$ be algebraically independent
	   rational functions in $f^* \mathbb C(B) \subset \mathbb C(X)$.  Notice that item (3) of Lemma \ref{L:preparation}
       guarantees that $\omega_3 \wedge df_1 \wedge \ldots \wedge df_q= 0$. Assume that the foliation
       $\omega_0\wedge df_1\wedge\cdots\wedge df_q$ has some transcendental
       leaf. We want to prove that $\mathcal F$ is transversely projective.
       The proof is recursive on $q$: we first prove that, on each fixed codimension
       $q-1$ subvariety defined by (an irreducible component of ) $f_2,\ldots,f_q=\text{constant}$, the restriction
       of $\mathcal F$ is transversely projective. Note that this later foliation must
       have transcendental leaves also so that we can go on with codimension $q-2$
       subvarieties defined by $f_3,\ldots,f_q=\text{constant}$. It is thus enough
       to consider the case $q=1$.

       We start with the easier case where $\mathcal F$ restricts as a transversely
       Euclidean foliation on the level sets. In other words, assume that
       $$\omega_1\wedge df=0.$$
       Then we can write $\omega_1=gdf$ for some rational function $g$
       and derivating $d\omega_0=\omega_0\wedge\omega_1$, we get
       $$\omega_0\wedge df\wedge dg=0.$$
       If $df\wedge dg=0$, then $d\omega_1=0$ so that $\mathcal F$ is transversely affine. If $df\wedge dg\not=0$,
       then $g$ is a rational first integral for $\mathcal F\vert_{f=\text{constant}}$.

       Before proceeding  to the next case, note that our assumption in the previous case is equivalent
       to $d\omega_0\wedge df=0$. Indeed, this means that $\omega_0\wedge\omega_1\wedge df=0$ and by division (recall that $\omega_0\wedge df\neq 0$) $\omega_1=gdf+h\omega_0$
       for rational functions $g$ and $h$. But (see \cite[Section 2.1]{MR2324555}) we can then replace
       $\omega_1$ by $\omega_1-h\omega_0$ without changing the equality
       $d\omega_0=\omega_0\wedge\omega_1$.

       We will treat the  affine case, now assuming
       $$\omega_2\wedge df=0.$$
       One can write $\omega_2=gdf$ as before and we get $d\omega_1=g\omega_0\wedge df$.
       If $g=0$ then $d\omega_1=0$ and $\mathcal F$ is transversely affine. If $g\not=0$,
       by derivation we get
       $$\left(\omega_1-\frac{dg}{g}\right)\wedge\omega_0\wedge df=0.$$
       Substituting $\omega_0=g\cdot\tilde\omega_0$ and $\omega_1=\tilde\omega_1+\frac{dg}{g}$, we get (the same conclusion with $g=1$)
       $$\tilde\omega_1\wedge\tilde\omega_0\wedge df=0$$
       and thus $d\tilde\omega_0\wedge df=0$ and we are back to the Euclidean case.

       We finally end with the general projective case, assuming
       $$\omega_3\wedge df=0.$$
       Writing $\omega_3=gdf$, we get $d\omega_2=g\omega_0\wedge df+\omega_1\wedge\omega_2$. If $g=0$, then $\mathcal F$ is transversely projective; if not, after derivation
       we get
       $$(\omega_1-\frac{1}{2}\frac{dg}{g})\wedge \omega_0\wedge df=0.$$
       After division, we find
       $$\omega_1=\frac{1}{2}\frac{dg}{g}+hdf+k\omega_0$$
       for rational functions $h$ and $k$. In fact, we can replace $\omega_1$ by $\omega_1-k\omega_0$ thus assuming $k=0$ (we now forget $\omega_2$).
       If $dh\wedge df=0$, then $\omega_1=\frac{1}{2}\frac{dg}{g}+hdf$ is closed and $\mathcal F$ is transversely affine. If $dh\wedge df\not=0$, then
       $$\omega_0\wedge dh\wedge df=\omega_0\wedge d\omega_1=0$$
       and $h$ is a first integral for $\mathcal F\vert_{f=\text{constant}}$.
    \end{proof}

\begin{lemma}\label{L:persistence}
	Let $\F$ be a transversely affine foliation  of codimension one on a projective manifold $X$.
	Assume there exists a projective surface $\Sigma \subset X$ for which the restriction of $\F$
    to $\Sigma$ is  virtually transversely Euclidean and admits no rational first integral then the same holds true for $\F$.
\end{lemma}
\begin{proof}
	This is a  consequence of the structure theorem for transversely affine foliations, see
	\cite{MR3294560}   and  \cite[Theorem D]{MR3522824}, which says that transversely affine foliations on projective manifolds which are
    not virtually transversely Euclidean are pull-backs of Riccati foliations on surfaces. Aiming at a contradiction,
    assume that $\F$ is transversely affine but not transversely Euclidean. Thus there exists a rational
	map $\varphi : X \dashrightarrow  S$ to a ruled surface, and a Riccati foliation on $S$ such that $\varphi^* \mathcal R = \F$.
    The restriction of $\varphi$ to $\Sigma$ must be dominant as otherwise $\F_{|\Sigma}$ would
	admit a rational first integral. Thus, if we assume that $\F_{|\Sigma}$ is virtually transversely
    Euclidean then  Lemma \ref{L:wellknown} implies that the same holds true for $\mathcal R$.
    But this would imply that $\F= \varphi^* \mathcal R$ is also virtually transversely Euclidean contrary to our assumptions.
    This gives the sought contradiction which proves the lemma.
\end{proof}


\section{Tangential foliation}\label{S:tangential}

This section is devoted to the proof of Theorem \ref{THM:iff}. We start by recalling the definition of the tangential foliation. For a more detailed account, we invite the reader to consult
\cite[Section 6]{2011arXiv1107.1538L}.

\subsection{Tangential foliation}

Let $X$ be a uniruled projective manifold. These are characterized by the existence of  non-constant morphisms $f : \mathbb P^1 \to X$ such that $f^* \TF$ is generated by global sections, the so called free morphisms. At a neighborhood of a free morphism the scheme $\Mor(\mathbb P^1, X)$ is smooth and the tangent space of $\Mor(\mathbb P^1 ,X)$ at the  point $f$ is naturally isomorphic to $H^0(\mathbb P^1,f^* T_X)$.

Let us fix a foliation $\F$ on  $X$ and  an irreducible component $M \subset \Mor(\mathbb P^1,X)$ of the space of morphism from $\mathbb P^1$ to $X$ containing a free morphism $f$.
Recall from \cite[Section 6]{2011arXiv1107.1538L} that $\F$ induces a natural foliation on $M$, called the tangential foliation of $\F$ and denoted by $\Ftang$,
with tangent space at a general free morphism $f$ given by $H^0(\mathbb P^1, \TF) \subset H^0(\mathbb P^1, T_X)$.

The tangential foliation admits the following alternative description. If
$\ev : M \times \mathbb P^1 \to X$ is the evaluation morphism, $p : M \times \mathbb P^1 \to M$ is the natural projection, and
$\mathcal H$ is  the  foliation on $M \times \mathbb P^1$  defined by the projection to $\mathbb P^1$ then
 $\Ftang$ coincides with  $p_* (\mathcal H \cap \ev^* \F)$.

\subsection{Foliations on $B\times \mathbb P^1$}
Let $L$ be a leaf of $\Ftang$ and $\overline L$ its Zariski closure.
We recall below  \cite[Theorem 6.5]{2011arXiv1107.1538L} which describes  the  foliation $\G= \ev^* \F_{|\overline L \times \mathbb P^1}.$

\begin{thm}\label{T:6.5}
Let $B$ be an algebraic manifold, let $\mathcal G$ be a codimension one foliation on $B \times \mathbb P^1$,
let $\pi: B \times \mathbb P^1 \to B$ be the natural projection, and let $\mathcal H$ be the codimension one foliation defined by
the fibers of the other natural projection $\rho:  B  \times \mathbb P^1  \to \mathbb P^1$.
If the general fiber of $\pi$ is generically transverse to $\mathcal G$ and the general leaf of the direct image $\mathcal T= \pi_*(\mathcal G \cap \mathcal H)$
is Zariski dense  then the codimension of $\mathcal T$ is at most three. Moreover,
\begin{enumerate}
    \item If $\codim \mathcal T =1$ then $\mathcal G$ is defined by a closed rational $1$-form;
    \item If $\codim \mathcal T=2$ then $\mathcal G$ is transversely affine;
    \item If $\codim \mathcal T=3$ then $\mathcal G$ is transversely projective.
\end{enumerate}
\end{thm}

\begin{remark}\label{R:6.5bis}
The proof presented in \cite{2011arXiv1107.1538L}  gives more information as we now proceed to recall. It starts by
showing that  the foliation $\G$ is defined by a rational $1$-form
which can be written as
\[
    \omega =  dz + a(z)  \alpha + b(z)  \beta + c(z)  \gamma
\]
where $\alpha, \beta, \gamma$ are rational $1$-forms on $B$; $a, b, c$ are rational functions on $\mathbb P^1$; and $z$ is a coordinate on $\mathbb P^1$.
To avoid overburdening the notation, we identify $1$-forms on $B$ with their pull-backs through $\pi$ to $B \times \mathbb P^1$,  and we also identify $1$-forms
on $\mathbb P^1$ with their pull-backs through $\rho : B \times \mathbb P^1 \to \mathbb P^1$ .
Furthermore, the  codimension of $\mathcal T$ coincides with the dimension of the $\mathbb C$-vector space generated by $\alpha, \beta, \gamma$.

When $\codim \mathcal T = 1$ then the foliation $\G$ is induced by the sum of pull-backs under the natural projections of closed rational $1$-form on $B$ and on $\mathbb P^1$.
To wit, $\G$ is defined by  $ \frac{dz}{f(z)} + \eta$, where $f$ is rational function on $\mathbb P^1$ and $\eta$ is closed
rational $1$-form on $B$.  Notice that  the zero set of this  closed rational $1$-form $\omega$ has no divisorial components.

If instead  $\codim \mathcal T \ge 2$ then the proof of  \cite[Theorem 6.5]{2011arXiv1107.1538L} establishes  the existence of a rational map
$\varphi : \mathbb P^1 \to \mathbb P^1$ and a Riccati foliation $\mathcal R$ on the manifold $B \times \mathbb P^1$ such that $\G = (\id_B \times \varphi)^* \mathcal R$.
The case where $\codim \mathcal T=3$ is distinguished from the case where $\codim \mathcal T=2$ by the existence of fibers of the projection
$\rho:B \times \mathbb P^1 \to \mathbb P^1$ invariant by $\G$. When $\codim \mathcal T=3$ none of the $\rho$-fibers are invariant. When $\codim \mathcal T=2$ the
Riccati foliation $\mathcal R$ has exactly one invariant $\rho$-fiber and, consequently, the  foliation $\G$ has at least one invariant $\rho$-fiber.
\end{remark}

We proceed to investigate the converse of Theorem \ref{T:6.5}. We start with a simple observation.

\begin{lemma}\label{L:horizontal}
    Notation and assumptions as in Theorem \ref{T:6.5}.
    If $Y$ is an algebraic leaf of $\G$ generically transverse to the fibers of $\pi$  then $Y$ is contained in a fiber of $\rho$.
\end{lemma}
\begin{proof}
	By the definition of $\Ftang$ the restriction of $Y$ to $L \times \mathbb P^1$ is horizontal.
    Therefore the { restriction to $Y \cap (L\times \mathbb P^1)$} of the projection $\overline L \times \mathbb P^1 \to \mathbb P^1$ is locally constant.
    Since $L$ is Zariski dense { $\overline L=B$} the same holds true over $Y$ itself. Therefore $Y$ is horizontal.
\end{proof}

\begin{prop}\label{P:transEucl}
Notation and assumptions as in Theorem \ref{T:6.5}.
If $\G$ is virtually transversely Euclidean then $\codim \mathcal T=1$.
\end{prop}
\begin{proof}
  Aiming at a contradiction, let us assume that $\G$ is virtually transversely Euclidean but $\delta= \codim \mathcal T \ge 2$. According to Remark \ref{R:6.5bis} we can assume that $\G$ is the
  pull-back of a Riccati foliation $\mathcal R$ on $B \times \mathbb P^1$ through a rational map of the form $\id_B \times \varphi : B \times \mathbb P^1 \dashrightarrow B \times \mathbb P^1$. In particular, $\mathcal R$ is defined by a rational $1$-form which can be written as
  \[
    \omega =  dz +   \alpha + z \beta + z^2  \gamma
    \]
  where $\alpha, \beta, \gamma$ are rational $1$-forms on $B$ which define the foliation $\mathcal T$ and $z$ is a coordinate on $\mathbb P^1$.   The integrability condition $\omega \wedge d \omega=0$ implies that
  \begin{align*}
    d \alpha &= \alpha \wedge \beta \\
    d \beta  &= 2\alpha \wedge \gamma \\
    d \gamma &= \beta \wedge \gamma \, .
  \end{align*}
  In particular, the assumption $\codim \mathcal T\ge2$ implies $d \omega\neq 0$.

  Since we are assuming that $\mathcal G$ is virtually transversely Euclidean, the same holds true for $\mathcal R$ according to Lemma \ref{L:wellknown}. Therefore there exists a non-zero closed logarithmic $1$-form $\omega_1$ on $B \times \mathbb P^1$ with periods commensurable to $\sqrt{-1} \pi$ such that $d \omega = \omega \wedge \omega_1$. Since any holomorphic function on $\mathbb P^1$ is constant, the {\it canonical} multi-valued first integral
  \[
    F = \int \left(\exp\int \omega_1\right) \omega_0
  \]
  must have a hypersurface contained in its singular set dominating the basis of the fibration $\pi : B \times \mathbb P^1 \to B$. As this hypersurface is clearly invariant by $\mathcal R$, we can apply Lemma \ref{L:horizontal} to guarantee that, after a change of coordinates on $\mathbb P^1$, the section $\{ z= \infty\}$ is invariant by $\mathcal R$. In these new coordinates, $\gamma =0$, i.e. $\mathcal R$ is defined by $\omega = dz + \alpha + z \beta$. This already shows that $\codim \mathcal T \le 2$, since
  $\mathcal T$ is the foliation defined by $\alpha$ and $\beta$.

   The integrability condition implies that $(\omega, \beta)$ is a transverse affine structure for $\mathcal R$. Therefore both $(\omega, \beta)$ and $(\omega, \omega_1)$ are transverse structures for $\mathcal R$. If they are the same then  $\omega_1=\beta$ and
  $\exp( n \int \beta)$ is a rational function constant along the leaves of $\mathcal T$ for any sufficiently divisible integer $n$ ($\omega_1$ is logarithmic with periods commensurable to $\sqrt{-1} \pi$). As we are assuming that the leaves of $\mathcal T$ are Zariski dense, we deduce that $\beta =0$ and, consequently, $\codim \mathcal T=1$ contrary to our assumptions.

  We can assume that  $\omega_1$ and $\beta$ are distinct closed rational $1$-forms. Their difference is a non-zero closed rational $1$-form defining $\mathcal R$. Therefore we can write $\beta - \omega_1 = h \omega$ for a certain non-constant rational $h \in \mathbb C(B \times \mathbb P^1)$. The irreducible components of the  divisor of zeros and poles of $h$, according to Lemma \ref{L:horizontal}, are either fibers of $\pi: B \times \mathbb P^1 \to B$ or fibers of $\rho : B\times \mathbb P^1 \to \mathbb P^1$. Therefore we can write $h$ as a product of $a \in \mathbb C(B)$ with $b \in \mathbb C(\mathbb P^1)$.

  We claim that the function $a \in \mathbb C(B)$  is constant. Indeed, from the equalities $d(h \omega) =0$ and
  $dz \wedge db =0$  we deduce that
  \[
    dz \wedge d ( a \omega) =0  \implies dz \wedge da \wedge \omega + a dz \wedge d\omega = 0 \, .
  \]
  Using that $d\beta=0$ we can writing
  \[
    dz \wedge da \wedge(\alpha + z \beta ) + a dz \wedge d \alpha =0\, .
  \]
  Finally, taking the wedge product of this last identity with $\beta$ we conclude that $dz \wedge da \wedge \alpha \wedge \beta =0$. It follows that $a$ is a first integral for $\mathcal T$, and as such must be constant.

  Now from $d (b\omega)= 0$ for $b \in \mathbb C(\mathbb P^1)$, we deduce the identity $db \wedge \alpha + b d \alpha + d(zb) \wedge\beta +zhd\beta =0$. After taking the wedge product with $\alpha$ and using the vanishing of $d \beta$ and of $\alpha \wedge d\alpha$, we conclude that $\alpha\wedge \beta=0$. This implies $\codim \mathcal T \le 1$. The proposition follows.
\end{proof}

Arguing as in the beginning of the proof above, one also obtains  the following result.

\begin{cor}\label{C:transAff}
Notation and assumptions as in Theorem \ref{T:6.5}.
If $\G$ is  transversely affine then $\codim \mathcal T\le2$.
\end{cor}

\subsection{Synthesis (Proof of Theorem \ref{THM:iff})}\label{SS:non-algebraic}
Let  $X$ be a uniruled projective manifold and $\F$ be a codimension one foliation on $X$.
We  fix an irreducible component $M$ of the space of morphisms from $\mathbb P^1$ to $X$  containing a free morphism and
let $\Ftang$  stand for the tangential foliation of $\F$ defined on this irreducible component $M$.
We  denote by  $\overline \Ftang$  the foliation on $M$ with general leaf given by the Zariski closure of a leaf of $\Ftang$. The existence of
a foliation with these properties follows from \cite{MR2223484}.

\begin{thm}[Theorem \ref{THM:iff} of the Introduction]\label{T:iff}
    Let $\delta = \dim \overline{\Ftang} - \dim \Ftang$.
    If the general leaf of $\Ftang$ is not algebraic (i.e. $\delta>0$)  then $\delta \le 3$. Furthermore the following assertions hold true.
	\begin{enumerate}
		\item $\delta=3$ if and only if $\F$ is transversely projective but not transversely affine; and
		\item $\delta=2$ if and only if $\F$ is transversely affine but not virtually transversely Euclidean; and
		\item $\delta=1$ if and only if $\F$ is virtually transversely Euclidean.
	\end{enumerate}
\end{thm}
\begin{proof}
	Let $L$ be a general leaf of $\Ftang$ which we will assume not algebraic. If $\overline L$ is the  Zariski closure of $L$ and $U$
is the smooth locus of $\overline L$
	then we are in position to apply Theorem \ref{T:6.5} to $\G = (\ev^*\F)_{|U\times \mathbb P^1}$ and deduce
    that $\G$ is transversely projective and  $ \dim \overline L - \dim L = \delta \in \{ 1, 2, 3\}$. Theorem \ref{T:extension} implies that $\ev^* \F$ is also transversely projective.

    If the general leaf of $\Ftang$ is not algebraic  then combining  Theorem \ref{T:6.5}  with  Theorem \ref{T:extension} one deduces
that $\F$ is transversely projective.

	Let us first prove assertion (3). If $\delta  =1$ then $\G$ is defined by a closed rational $1$-form, see Remark \ref{R:6.5bis}.
    Theorem \ref{T:extension} implies that $\ev^* \F$ is virtually transversely Euclidean. Lemma \ref{L:wellknown} implies that the
    same holds true for $\F$. Reciprocally, if $\F$ is virtually transversely Euclidean then the same holds true for $\G$. We apply Proposition \ref{P:transEucl} to deduce that $\delta=1$.
    Assertion (3) follows.

    The proof of assertion (2) is similar. If $\delta=2$ then $\G$ is transversely affine. Theorem \ref{T:extension} implies that $\ev^* \F$ is transversely affine and Lemma \ref{L:wellknown} implies that the same holds true for $\F$. Reciprocally, if $\F$ is transversely affine then Corollary \ref{C:transAff} implies $\delta \le 2$. Assertion (3) implies $\delta \ge 2$ and assertion (2) follows.

    As before, Theorem \ref{T:extension} combined with Lemma \ref{L:wellknown} imply that $\F$ is transversely projective. Reciprocally, as we already know that $\delta \le 3$, assertions (2) and (3) imply assertion (1).
\end{proof}

\subsection{Bootstrapping (Proof of Theorem \ref{THM:bootstrapping})}

We  proceed to prove Theorem \ref{THM:bootstrapping}. We keep the notations from Section \ref{SS:non-algebraic}.

\begin{lemma}\label{L:transversely affine}
    Assume  that the general leaf of $\Ftang$ is not algebraic.
    If $\F$ has an algebraic leaf $Y \subset X$ which intersects the image of a general morphism $f \in M$ then
    $\F$ is transversely affine.
\end{lemma}
\begin{proof}
	Let $L$ be a general leaf of $\Ftang$; $\overline L$ be its Zariski closure; and $U\subset \overline L$ be the smooth locus of the Zariski closure.
    Lemma \ref{L:horizontal} implies  that $\G = \ev^* \F_{|U \times \mathbb P^1}$ has a horizontal leaf.
    Since $\G = (\id_{U}\times \varphi)^* \mathcal R$ it follows that the Riccati foliation $\mathcal R$ also has a horizontal leaf.
    Thus, in a suitable coordinate system where the horizontal is at $\{ z = \infty\}$, the Riccati foliation   $\mathcal R$ is defined by
	\[
		dz + \omega_0 + z \omega_1 \, .
	\]
	Hence the restriction of $\Ftang$ to $\overline L$ is defined by $\omega_0$ and $\omega_1$ and
    therefore $\dim \overline \Ftang - \dim \Ftang \le 2$. We  apply Theorem \ref{T:iff} to conclude that $\F$ is transversely affine.
\end{proof}

\begin{lemma}\label{L:Q-Euclidean}
	Assume $X$ is simply connected and that the general leaf of $\Ftang$ is not algebraic.
    If the image of a general morphism $f \in M$ intersects non-trivially  every algebraic leaf of $\F$ then
    $\F$ is virtually transversely Euclidean.
\end{lemma}
\begin{proof}
	We keep the notation from Lemma \ref{L:transversely affine}. According to Theorem \ref{T:iff} it suffices to show that $\G = \ev^* \F_{|U\times \mathbb P^1}$ is defined by a closed rational $1$-form.
	
	First notice that $\F$ must have some algebraic leaf. Otherwise the existence of a transverse structure for $\F$ given by Theorem \ref{T:iff}
	and the simple connectedness of $X$ would give rise to the existence of a rational first integral for $\F$. It follows that the leaves of $\Ftang$ are algebraic  contradicting our initial assumption.
    Therefore we can apply Lemma \ref{L:transversely affine} to deduce that $\F$ is a transversely affine foliation. Fix rational $1$-forms $\omega_0$ and $\omega_1$ on $X$ such that
    $d\omega_0 = \omega_0 \wedge \omega_1$,  $d \omega_1=0$, and $\omega_0$ defines $\F$.
	
	If $\F$ is virtually transversely Euclidean then the result follows from Theorem \ref{T:iff}.
	Assume from now on that $\F$ is not virtually transversely Euclidean. Since $X$ is simply connected, this implies that
	$\omega_1$ is not a logarithmic differential with rational residues. Notice that $\omega_0$ and $\omega_1$ are not uniquely defined, but any different pair $(\omega_0',\omega_1')$ satisfying the conditions above will
	satisfy $\omega_0 = h \omega_0'$ and $\omega_1 = \omega_1' + d\log h$. In particular, the fact that $\omega_1$ is not a logarithmic $1$-form with rational residues does not depend on the choice of the pair.
	
	Aiming at a contradiction assume that $\dim \overline L - \dim L \ge 2$. According to Remark \ref{R:6.5bis}, the foliation  $\G$
    is the pull-back of a Riccati foliation on $U\times \mathbb P^1$ leaving invariant the section at infinity. A simple computation shows
    that this implies the existence of a transverse affine structure for  $\G$  defined by a pair $(\alpha_0,\alpha_1)$ such that the restriction of
    $\alpha_1$  at  a general fiber $\mathbb P^1$ is logarithmic with integral residues. But
    the pull-back of $\omega_1$  under the evaluation morphism  is not of this form since the image of $f$ intersects all components of the polar divisor of $\omega_1$. Hence $\G$ admits two non-equivalent transversely affine structures. Lemma \ref{L:uniqueness} implies that $\G$ is transversely Euclidean.
    The lemma follows from Theorem \ref{T:extension}.
\end{proof}

\begin{thm}[Theorem \ref{THM:bootstrapping} of the Introduction]
If $X$ is simply connected, the general leaf of $\Ftang$ is not algebraic, and  the general morphism $f \in M$ intersects non-trivially and transversely every algebraic hypersurface invariant by $\F$ then $\F$ is defined by a closed rational $1$-form
without divisorial components in its zero set.
\end{thm}
\begin{proof}
  We start by showing that $\F$ is defined by a closed rational $1$-form.  Lemma \ref{L:Q-Euclidean} implies that $\F$ is virtually transversely Euclidean. Therefore the transverse structure for $\F$ is defined by a pair $(\omega_0, \omega_1)$
  where $\omega_1$ is a closed logarithmic $1$-form with periods commensurable to $\pi \sqrt{-1}$. If the periods of $\omega_1$ are integral multiples of $2 \pi \sqrt{-1}$ then
  \[
     \exp\left(\int \omega_1\right) \omega_0
  \]
  is the sought closed rational $1$-form.
  Assume from now on that the periods of $\omega_1$ are not integral multiples of $2 \pi \sqrt{-1}$.

  Let $L$ be a general leaf of $\Ftang$ and $U \subset \overline L$ be the smooth locus of its Zariski closure. Let $\G = \ev^* \F_{|U\times \mathbb P^1}$ be the pull-back of $\F$ to $M \times \mathbb P^1$ under the evaluation morphism. On the one hand, according to Remark \ref{R:6.5bis}, $\G$ is defined by a closed rational $1$-form. On the other hand, the
  transversality of the general $f \in M $ with the $\F$-invariant algebraic hypersurfaces imply that the  periods of $\ev^* \omega_1$ are also not integral multiples of $2 \pi \sqrt{-1}$.
  Hence $\G$ admits two non-equivalent virtually transversely Euclidean structures. Lemma \ref{L:uniqueness} implies that all the leaves of $\G$ are algebraic. It follows that all leaves of $\Ftang$ are algebraic contrary to our assumptions. This concludes the proof that $\F$ is defined by a closed rational $1$-form $\omega$.

  To verify that $\omega$ does not have divisorial components in its zero set, we proceed similarly. Pull-back $\omega$ to $U\times \mathbb P^1$ using the evaluation morphism. Any irreducible divisorial component of the zero set of $\omega$ is a  $\F$-invariant algebraic hypersurface. Therefore, by assumption, the restriction of the $1$-form $\ev^*\omega$  on $U \times \mathbb P^1$ will have an  irreducible component in its zero set which dominates $U$. This is only possible if $\ev^* \omega$ does not depend on the variables of $U$, i.e. $\ev^* \omega = a(z) dz$ where $z$ is a coordinate on $\mathbb P^1$ and $a \in \mathbb C(\mathbb P^1)$ is a rational function. It follows that $\overline L$ and $L$ have the same dimension, contrary to our assumptions. Theorem \ref{THM:bootstrapping} follows.
\end{proof}

\subsection{Tangential foliation with algebraic leaves} When  $\Ftang$ is a foliation by algebraic leaves, the original foliation $\F$ inherits a subfoliation by algebraic leaves.
We keep the notation settled at the beginning of Section \ref{SS:non-algebraic}.

\begin{prop}\label{P:algleaves}
	If all the leaves of $\Ftang$ are algebraic, then there exists
	a foliation by algebraic leaves $\mathcal A$ contained in $\F$ such that $H^0(\mathbb P^1, f^* T_{\mathcal A}) = H^0(\mathbb P^1,f^* \TF)$ for any sufficiently general morphism $f \in M$. In particular, $\F$ is the pull-back under a rational map of a foliation on a lower dimensional manifold.
\end{prop}
\begin{proof}
	Let $\mathcal A$ be the maximal foliation by algebraic leaves contained in $\F$. The existence of such
	$\mathcal A$ is assured by \cite[Lemma 2.4]{2011arXiv1107.1538L}. If $L$ is a general leaf of $\Ftang$ then
	the image of $L\times \{ z\}$ (for any fixed $z \in \mathbb P^1$) under the evaluation morphism $\ev : M \times \mathbb P^1 \to X$  is contained in a leaf of $\F$. Moreover, it is also contained in a leaf of $\mathcal A$.
    This makes clear that $\mathcal A_{\rm tang}$ contains $\Ftang$. But since $\mathcal A$ is contained in $\F$, we must have $\mathcal A_{\rm tang}$ contained in
	$\Ftang$. Hence $\Ftang= \mathcal A_{\rm tang}$ and the equality $H^0(\mathbb P^1, f^* T_{\mathcal A}) = H^0(\mathbb P^1,f^* \TF)$ holds true for any
	sufficiently general morphism $f \in M$.
\end{proof}


\section{Foliations of degree three}\label{S:3}

\subsection{Proof of Theorem \ref{THM:A}}
Let $\F$ be a codimension one foliation of degree three on $\mathbb P^3$. The canonical bundle of $\F$ is  $\KF = \mathcal O_{\mathbb P^3}(1)$. Therefore, if $f: \mathbb P^1 \to \mathbb P^3$ is the parametrization of a general line then $h^0(\mathbb P^1, f^* T\F) \neq 0$. If $M \subset \Mor(\mathbb P^1, \mathbb P^3)$ is the irreducible component containing $f$, then $\Ftang$ is a foliation of positive dimension on $M$.

If the general leaf of $\Ftang$ is not algebraic then Theorem \ref{THM:bootstrapping} guarantees that $\F$ is defined by a closed rational $1$-form without divisorial components in its zero set. Thus $\F$ is described by item (1).

If the general leaf of $\Ftang$ is algebraic then let $\mathcal A\subset \mathcal F$ be the foliation by algebraic leaves given by
Proposition \ref{P:algleaves}. In particular, $h^0(\mathbb P^1, f^* T \mathcal A) = h^0(\mathbb P^1, f^* T\F)\neq 0$.

If $\dim \mathcal A =2$ then all the leaves of $\F$ are algebraic and Darboux-Jouanolou Theorem guarantees that $\F$ admits a rational first integral. The foliation $\F$ fits into the description given by item (4).

 If instead $\dim \mathcal A =1$ then $T\mathcal A$ is a line bundle and $h^0(\mathbb P^1, f^*T\mathcal A)\neq 0$ implies that $T\mathcal A = \mathcal O_{\mathbb P^3}$ or $T\mathcal A = \mathcal O_{\mathbb P^3}(1)$. In the first case, $\mathcal A$ has degree one and $\F$ is described by item (2). In the second case, $\mathcal A$ has degree zero and its leaves are lines through a unique point of $\mathbb P^3$. The foliation $\F$ is described by item (3). \qed

\bibliography{references}{}
\bibliographystyle{plain}
\end{document}